\documentclass[11pt,reqno]{amsart}

\usepackage{amssymb,bm,mathrsfs}

\usepackage{enumerate}
\usepackage{tikz}
\usepackage[centering,width=6in]{geometry}
\usepackage{mathtools}

\usepackage{tikz}
\usetikzlibrary{decorations.pathreplacing}

\usepackage[normalem]{ulem}% to strike out \sout{}

\usepackage{graphicx}
 \usepackage{color}
  \definecolor{blue}{rgb}{0,0,1}

 \definecolor{red}{rgb}{1,0,.2}

\usepackage{calc}
\usepackage{accents}

\newtheorem{theoremvar}{Theorem}
\newenvironment{theorem*}[1]
 {\renewcommand{\thetheoremvar}{#1}\theoremvar}
 {\endtheoremvar}

 \newtheorem{corollaryvar}{Corollary}
\newenvironment{corollary*}[1]
 {\renewcommand{\thecorollaryvar}{#1}\corollaryvar}
 {\endtheoremvar}

\newtheorem{theorem}{Theorem}[section]
\newtheorem{lemma}[theorem]{Lemma}

\newtheorem{corollary}[theorem]{Corollary}
\newtheorem{question}[theorem]{Question}
\newtheorem{conjecture}[theorem]{Conjecture}
\theoremstyle{definition}
\newtheorem{definition}[theorem]{Definition}
\newtheorem{example}[theorem]{Example}

\newcommand{\R}{\mathbb{R}}

\newcommand{\Z}{\mathbb{Z}}
\newcommand{\N}{\mathbb{N}}

\usepackage[framemethod=TikZ]{mdframed}
\theoremstyle{remark}
\usetikzlibrary{shadows}
\usepackage{amsthm}

\makeatletter

\usepackage[colorlinks,pdfstartview=FitH]{hyperref}

\hypersetup{
    colorlinks=true,
    citecolor=blue,
    filecolor=black,
    linkcolor=red,
    urlcolor=black
}

\usepackage{graphicx,amsmath,amssymb}

\newcommand\rsmraise[1]{%
  \ifx#1\displaystyle .8\else
    \ifx#1\textstyle .8\else
      \ifx#1\scriptstyle .6\else
        .45%
      \fi
    \fi
  \fi}
\usepackage{setspace}

\begin{document}

\title[Interior of distance trees over thin Cantor sets]{Interior of distance trees over thin Cantor sets}

\author{Yeonwook Jung}
\address{Department of Mathematics, University of California, Irvine, CA 92697, USA}
\curraddr{}
\email{yeonwoj1@uci.edu}
% \thanks{}

\author{Krystal Taylor}
\address{Department of Mathematics, The Ohio State University, Columbus, OH 43210-1174, USA}
\email{taylor.2952@osu.edu}
% \thanks{}
\begin{abstract}
It is known that if a compact set $E$ in $\mathbb{R}^d$ has Hausdorff dimension greater than $(d+1)/2$, then its $n$-chain distance set $$\Delta^n(E) = \left\{\left(\left|x^1-x^2\right|,\cdots, \left|x^{n}- x^{n+1}\right|\right)\in \mathbb{R}^n: x^i \in E, x^i\neq x^j \text{ for } i\neq j \right\}$$ has nonempty interior for any $n\in \mathbb{N}$. In this paper, we prove that for every Cantor set $K\subset \mathbb{R}^d$ and for every $n\in\mathbb{N}$, there exists $\widetilde{K}\subset \mathbb{R}^d$ such that the pinned $n$-chain distance set of $K\times \widetilde{K}\subset \mathbb{R}^{2d}$ has nonempty interior, and hence, that $\Delta^n(K\times \widetilde{K})$ has nonempty interior. Our results do not depend on the Newhouse gap lemma but rather on the containment lemma recently introduced by Jung and Lai. Our results generalize three-fold: to arbitrary finite trees, to higher dimensions, and to maps that have non-vanishing partials. As an application, we provide a class of examples of Cantor sets $E\subset \mathbb{R}^{2d}$ so that for any $s\geq d$, $\dim_{\rm H}(E)= s$ and $\Delta_x^n(E)^\circ{}\neq \varnothing$ for some $x\in E$.
\end{abstract}
\maketitle

\section{Introduction and Historical Background}
The seminal Falconer distance conjecture asks for a lower bound on the Hausdorff dimension of a compact set $E$ of $\R^d$ that guarantees the distance set 
$$\Delta(E) = \{ |x-y| : x,y\in E\}$$ has positive Lebesgue measure. Falconer proved that $\frac{d+1}{2}$ suffices, and he gave a constructive example to show that the conclusion need not hold for sets with Hausdorff dimension less than $\frac{d}{2}$ \cite{Falc85paper}. The best known result in the plane is due to Guth, Iosevich, Ou, and Wang and shows that $\dim_{\rm H} (E) >\frac54$ suffices \cite{GIOW}
even for the pinned distance sets defined below; for a summary recent improvements in higher dimensions, see the introduction in \cite{BFOP}. The Falconer distance problem also has a number of variants, including the study of interior of distance sets, pinned distance sets, and more intricate distance sets over graphs. 

\smallskip

For the interior of the distance set, Mattila and Sj\"olin proved that $\dim_{\rm H}(E) > \frac{d+1}{2}$ suffices to guarantee  that $\Delta(E)$ has an interior \cite{MS}. An extension and a short analytic proof of this result was provided subsequently by Iosevich, Mourgoglu, and Taylor in \cite{IMT12}, where the authors show that the Euclidean distance can be replaced by more general metrics. 

\smallskip

The \textit{pinned distance set} of a compact set $E\subset \R^d$ at a pin $x\in E$ is defined by $$\Delta_{x}(E) = \{ |x-y| : y\in E\}.$$ Peres and Schlag \cite{PS} proved that $\dim_{\rm H}(E)> \frac{d+2}{2}$ suffices to guarantee that $\Delta_x(E)$ has an interior point for some $x\in E\subset \R^d$ for $d\geq 3$.  A recent exciting improvement in dimensions $d=2$ and $d=3$ due to Borges, Foster, Ou, and Palsson \cite{BFOP} shows that for compact sets $E$:
\begin{enumerate}[(i)]
    \item if $E\subset \R^2$ and $\dim_{\rm H}(E)> \frac{7}{4}$, then $\Delta_{x}(E)^{\circ}\neq \varnothing$ for some $x\in E$ (see Figure \ref{interior number line});
\item if $E\subset \R^3$, then $\dim_{\rm H}(E)> \frac{12}{5}$ suffices to make the same conclusion. 
\end{enumerate}

\smallskip

\begin{figure}[h]
\begin{tikzpicture}
  % Draw the number line
  \draw (0,0) -- (4.5,0) node[right] {\,\,$\dim_{\rm H}(C_1\times C_2)$\,\,\,};
  
  \draw (0,0) -- (0,0.2) node[above] {1};
  
  \draw (4,0) -- (4,0.2) node[above] {2};
  
  \draw (3,0) -- (3,0.2) node[above] {$\frac{7}{4}$};
  
  \draw [decorate,decoration={brace,mirror,amplitude=6pt}]
    (3,-0.3) -- (4,-0.3) node[midway,below,yshift=-4pt]{$\Delta_x(E)^{\circ}\neq \varnothing$};
\end{tikzpicture}
\caption{If $E\subset \R^2$ is compact and $\dim_{\rm H}(E) >\frac74$, then the pinned distance set has nonempty interior for some $x\in E$. }
\label{interior number line}
\end{figure}
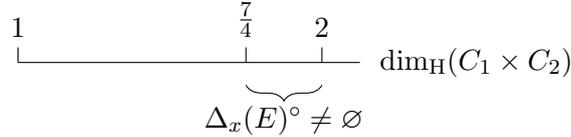

The \textit{$n$-chain distance set} of a compact set $E\subset \R^d$ is defined by \begin{equation*}\label{chain dist}
\Delta^n(E) := \left\{ (|x^1-x^2|, |x^2 - x^3|, \dots, |x^n - x^{n+1}|): x^i \in E, \,\, x^i \neq x^j \text{ for } i\neq j \right\}.
\end{equation*}
Bennett, Iosevich and Taylor prove in \cite{BIT} that $\Delta^n(E)$ has an interior point in $\R^n$ provided $\dim_{\rm H}(E) >\frac{d+1}{2}$. The results in \cite{BIT} also extend to trees (acyclic connected graphs) \cite{IT}. The proof technique in both \cite{BIT} and \cite{IT} relies on showing that the density on the distance set is a continuous function and the mapping properties of the spherical convolution operator $\sigma f*\mu$, where $\mu$ is a Frostman measure on $E$. 

\smallskip

The $n$-chain distance set was also investigated by Ou and Taylor in \cite{OT22}, where the authors proved that $\Delta^n(E)$ has positive Lebesgue measure provided $E\subset \R^2$ is a compact set satisfying $\dim_{\rm H}(E)>\frac54$. The proof leverages the result on pinned distance sets in \cite{GIOW}, and the proof technique also works in higher dimensions and extends to trees. 

\subsection{Improved results for product Cantor sets}\label{newhouse section}
There are known improvements to the aforementioned results involving the interior of pinned distance sets in the specific setting of product Cantor sets. Simon and Taylor, for instance, prove in \cite{simon_taylor} that if $C$ denotes a compact set of Newhouse thickness at least $1$, denoted $\tau(C)\geq 1$ (the middle-third Cantor set, for instance, satisfies this), then $\Delta_x(C\times C)$ has nonempty interior for any $x\in C\times C$. Using this result, it is possible to construct such a set $C\subset \R^2$ so that $\dim_{\rm H}(C)$ is arbitrarily close to $1$, obtaining a class of examples outside the reach of the result in \cite{BFOP} whose pinned distance sets have nonempty interior. 

\smallskip

For $n$-chain distance sets, McDonald and Taylor prove in \cite{McDonald-Taylor-finite-point-configurations} that $\Delta^n(K_1\times K_2)$, has nonempty interior provided $\tau(K_1)\cdot \tau(K_2)>1$. Further, the authors extend this result to include trees or finite acyclic graphs. This work is continued in \cite{McDonald_Taylor_2023_infinite_chains}, where distance sets corresponding to infinite trees and constant gap trees are investigated. In higher dimensions, Boone and Palsson use similar techniques to construct sets $C_a, C_b\subset \R^d$ so that $C= C_a\times C_b\subset \R^{2d}$ has Hausdorff dimension arbitrarily close to $ (2d-1)$ and $\Delta_x^n(C)$ has nonempty interior; see the examples in \cite[\S2]{BoonePalsson}. Also, Jung and Lai prove in \cite{jung-lai-interior-of-certain} that there exists a product Cantor set $K\subset \R^d$, for $d$ even, so that $\dim_{\rm H}(K) = \frac{d}{2}$ and $\Delta_x(K)$ has nonempty interior for some $x$ in $K$ \cite[Theorem 8.2]{jung-lai-interior-of-certain}. These results are summarized in Figure \ref{interior Cantor number line}.

\smallskip

We refer to \cite{Palis1993HyperbolicityAS, Yavicoli_Survey} or \cite{ST2025} for the definition and properties of Newhouse thickness.  Like Hausdorff dimension, Newhouse thickness has also played a role in the study of distance sets and finite point configurations \cite{BoonePalsson, Jiang, ST2025, yavicoli_patterns, Yavicoli_Gap_Lemma_Rd}.

\begin{figure}[h]
\begin{tikzpicture}
  \draw (0,0) -- (4.5,0) node[right] {\,\,$\dim_{\rm H}(E)$\,\,\,};

  % Mark 1
  \draw (0,0) -- (0,0.2) node[above] {$d$};

  % Mark 2
  \draw (4,0) -- (4,0.2) node[above] {$2d$};

\draw [decorate,decoration={brace,mirror,amplitude=6pt}]
    (0,-0.3) -- (4,-0.3) node[midway,below,yshift=-4pt]{$\Delta_x^n(C_1\times C_2)^{\circ{}}\neq \varnothing$};
\end{tikzpicture}
\caption{
For any $s>1$ and $n\geq 1$, there exists a Cantor set $C_1\times C_2\subset \R^{2d}$ so that $\dim_{\rm H}(C_1 \times C_2) =s$ and the pinned $n$-chain distance set has nonempty interior \cite{simon_taylor,BoonePalsson}. When $s=n=1$, this follows from \cite{jung-lai-interior-of-certain}. When $s=1$ and $n>1$, this follows from the present article.}
\label{interior Cantor number line}
\end{figure}

\subsection{The main idea of this paper}
A similarity between the results of McDonald and Taylor in \cite{McDonald-Taylor-finite-point-configurations, McDonald_Taylor_2023_infinite_chains} and those of Ou and Taylor in \cite{OT22} is that each of these results show that a sufficiently robust pinned distance result can be used to generate chains and trees. In particular, the \textit{pin wiggling lemma} in \cite{McDonald-Taylor-finite-point-configurations} shows that there exists an interval $I$ which is contained in $\Delta_x(K\times K)$, that remains in the pinned distance set if $x$ is perturbed by a small amount.

\smallskip

The main thrust of this paper is a new pin wiggling lemma based on the work of Jung and Lai. In particular, we build chains and trees in a product Cantor set setting and extend the nonempty interior results to \textit{pinned $n$-chain distance sets} defined by
\begin{align*}
\Delta_{x}^{n} (K\times \widetilde{K})
= \{ (|x-x^1|, |x^1-x^2|, \cdots, |x^{n-1} - x^n|) : x^i \in K\times \widetilde{K}, x^i \neq x^j \text{ for } i\neq j\},
\end{align*}
for $x\in K\times \widetilde{K}$. We also generalize our results to tree distance sets, which are defined in \S \ref{sec:main-results}.

\smallskip

More specifically, Jung and Lai prove that, given a Cantor set $K_1$ on $\R^d$, it is always possible to find another Cantor set $K_2$ so that the sum $g(K_1) +K_2$ (where g is a $C^1$ local diffeomorphism) has nonempty interior, and the existence of the interior is robust under small perturbation of the mapping. More generally, they show that the image set $H(z, K_1, K_2)$, where $H$ is some $C^1$ function on $\R^N \times \R^d\times \R^d$ with non-vanishing Jacobian, has nonempty interior for all $z$ in an open ball of $\R^N$ \cite[Theorem 1.1]{jung-lai-interior-of-certain}.

\smallskip
 
We apply this result with $$H(z, x, y) = \| (x,y) - z\|$$ to see that, given a set $K\subset \R$, there exists a Cantor set $\widetilde{K}\subset \R$ and an open neighborhood $U\subset \R^2$ so that $\bigcap_{z\in U} \Delta_z(K\times \widetilde{K})$ has interior. We then iteratively apply this result to build a set $\widetilde{K}$ for which $K\times \widetilde{K}$ has pinned chain and tree distance sets with interior. As a corollary, we prove that for any even $d$, there exists a Cantor set $K\subset \R^d$ so that the pinned chain and tree set of $K$ has nonempty interior for some $z\in K$. 

\smallskip
 
In previous works, chain and tree distance sets are built over selected skeleton points in $K\times \widetilde{K}$. The main hurdle in this work was that we cannot a priori choose skeleton points in $K\times \widetilde{K}$, since $\widetilde{K}$ is something we need to construct. Our solution is to choose the skeleton points from the diagonal line of $K\times K$, which then enables us to apply the above result to construct chains and trees and $\widetilde{K}$ simultaneously, and then we take a union of $K$ and $\widetilde{K}$.

\smallskip

Before formally stating our results, we end this section with a conjecture, which is a strengthening of Falconer's conjecture and is motivated by both the results stated above and the results of this paper. 
\begin{conjecture}[Interior variant of Falconer for pinned chain distance sets]
    Let $E\subset \R^d$ compact with $\dim_{\rm H}(E) >\frac{d}{2}$. Then, for every $n\in\N$, the pinned $n$-chain set has nonempty interior, $$\Delta^n_x(E)^{\circ{}} \neq \varnothing,$$
    for some $x\in E$. 
\end{conjecture}
Our Corollary \ref{cor:2-chain-hausdorff-dimension-zero} and its higher dimensional analogue, Corollary \ref{cor:n-chain-hausdorff-dimension-zero} both stated below, supply a class of examples where this conjecture holds in even dimensions. In particular, we show that for each $s\geq d$, there exists a product Cantor set $E\subset \R^{2d}$ satisfying $\dim_{\rm H}(E)= s$ and $\Delta_x^n(E)^\circ{}\neq \varnothing$ for some $x\in E$.  These results also hold for pinned trees, but we state this conjecture and the corollaries for pinned chains for simplicity. 

\smallskip

\subsection{Notations}
$|\cdot|$ denotes the Lebesgue measure, $(\cdot)^\circ{}$ denotes the interior, and $\text{conv}(\cdot)$ denotes the convex hull. A Cantor set is a subset of $\R^d$ that is compact, totally disconnected, and perfect. A sub-Cantor set of a Cantor set $C$ is a subset that is itself a Cantor set.

\smallskip
\subsection{Organization of the paper}
Our main results are presented in increasing level of complexity in 
\S\ref{sec:main-results}. 
The proof of our first pin wiggling lemma, Lemma \ref{lem:pin-wiggling}, is given in \S\ref{sec:lem}, and 
the proof Theorem \ref{thm:2-chain} on $2$-chains on subset of the plane is given in \S\ref{sec:2-chain}. 
The proof of Theorem \ref{thm:n-chain} on $n$-chains on subsets of $\R^2$ is given in \S\ref{sec:n-chain}. 
The proof of Theorem \ref{thm:tree} on $n$-trees on subsets of $\R^2$ is presented in \S\ref{sec:tree}. 
Our higher dimensional pin wiggling lemma, Lemma \ref{lem:pin-wiggling-Rd}, is presented in 
\S\ref{sec:lem-Rd}, and 
the proof of Theorem \ref{thm:tree-Rd} on finite trees on subsets of $\R^{2d}$ is presented in 
\S\ref{sec:tree-Rd}. 
We conclude our paper by discussing some open questions in
\S\ref{sec:open-questions}.

\smallskip
\section{Main results}\label{sec:main-results}
To build the ideas in this paper, we present our results in increasing levels of generality, starting with our result on the real line and progressing to results in higher dimensions. We first introduce Lemma \ref{lem:pin-wiggling} as our main tool in this paper; we also introduce its higher-dimensional analogue Lemma \ref{lem:pin-wiggling-Rd}. These can be seen as a variant of the pin wiggling lemma for distances in \cite[Lemma 3.5]{McDonald-Taylor-finite-point-configurations}.

\begin{lemma}[A variant pin wiggling lemma]\label{lem:pin-wiggling}
Let $K\subset \R$ be a Cantor set, $v=(v_1,v_2)\in \R^2$, and $I$ be a nonempty open interval such that $ K\cap I \neq \varnothing$. Then, there exist a Cantor set $\widetilde{K}\subset I$ with $|\widetilde{K}|>|K|$ and an open neighborhood $U_v\subset \R^2$ of $v$ such that
\begin{align*}
    \left(\bigcap_{z\in U_v} \Delta_z(K\times \widetilde{K})\right)^\circ \neq \varnothing.
\end{align*}
\end{lemma}
Lemma \ref{lem:pin-wiggling} is proved in \S \ref{sec:lem}. We now turn to applications. 

\subsection{The simplest case: 2-chains}
Our main result in the simplest setting is the following Theorem \ref{thm:2-chain}, which we prove in \S \ref{sec:2-chain}. As an application, we construct a Cantor set in $\R^2$ of Hausdorff dimension $1$ whose $2$-chain distance set has nonempty interior. Note that the proof can be easily adapted to obtain such Cantor sets of any Hausdorff dimension $s\geq 1$.

\begin{theorem}[2-chains on subsets of  $\R^2$]\label{thm:2-chain}\index{}
For every Cantor set $K\subset \R$, there exists a Cantor set $\widetilde{K}\subset \R$ with $|\widetilde{K}|>|K|$ such that $\left(\Delta^2_{y} (K\times \widetilde{K})\right)^\circ \neq \varnothing$ for some $y\in K\times \widetilde{K}$.
\end{theorem}
\begin{corollary}\label{cor:2-chain-hausdorff-dimension-zero}\index{}
There exists a Cantor set $C\subset \R^2$ such that $\dim_{\rm H}(C)=1$ and the pinned 2-chain distance set $\Delta^2_{y} (C)$ has nonempty interior for some $y\in C$.
\end{corollary}
\begin{proof}[Proof of Corollary \ref{cor:2-chain-hausdorff-dimension-zero}]
Let $K\subset \R$ be a Cantor set with zero Hausdorff dimension. Then, by Theorem \ref{thm:2-chain}, there exists a Cantor set $\widetilde{K}$ such that the pinned 2-chain distance set $\Delta^2_y (K\times \widetilde{K})$ has nonempty interior for some $y\in K\times \widetilde{K}$. We claim that $C=K\times \widetilde{K}$ is the desired Cantor set. Note that since $\widetilde{K}$ always has a positive Lebesgue measure, $\mbox{dim}_H(\widetilde{K}) = \mbox{dim}_B(\widetilde{K}) =1$. By the dimension estimate formula (\cite[Equation 7.2 and 7.3]{falconer-fractal-geometry})
\begin{align*}
    \dim_{\rm H}(K)+\dim_{\rm H}(\widetilde{K}) \leq \dim_{\rm H}(K\times \widetilde{K}) \leq \dim_{\rm H}(K)+\dim_B (\widetilde{K}),
\end{align*}
we have $\dim_{\rm H}(K\times \widetilde{K})=1$. This completes the proof.
\end{proof}

\subsection{$n$-chains and finite trees}
The case for $n$-chains is a direct generalization of the $2$-chain case. as given by the following Theorem \ref{thm:n-chain}, which we prove in \S \ref{sec:n-chain}. By the same argument in Corollary \ref{cor:2-chain-hausdorff-dimension-zero}, we construct a Cantor set in $\R^2$ of Hausdorff dimension 1 whose $n$-chain distance set has nonempty interior as an application.
\begin{theorem}[$n$-chains over subsets of $\R^2$]\label{thm:n-chain}\index{}
For every Cantor set $K \subset \mathbb{R}$ and every $n\in \N$, there exists a Cantor set $\widetilde{K} \subset \R$ with $|\widetilde{K}|>|K|$ such that $\left( \Delta^n_{y} (K \times \widetilde{K} ) \right)^\circ \neq \varnothing$ for some $y\in K\times \widetilde{K}$.
\end{theorem}
\begin{corollary}[]\label{cor:n-chain-hausdorff-dimension-zero}\index{}
For every $n\in\N$, there exists a Cantor set $C\subset \R^2$ such that $\dim_{\rm H}(C)=1$ and the pinned $n$-chain distance set $\Delta^n_{y} (C)$ has nonempty interior for some $y\in C$.
\end{corollary}
To state our result in the most general case of arbitrary finite trees, we introduce the necessary definitions. A \textit{tree} is a connected acyclic graph, a \textit{finite tree} is a tree whose vertex set is finite, and a \textit{rooted tree} is a tree with a designated root. The \textit{depth of a vertex} of a rooted tree is the length of the path from the root to the vertex, and the \textit{depth of a rooted tree} is the maximum depths among all vertices. Given a finite tree $T$ with $n$ vertices (and hence $n-1$ edges), the \textit{tree distance set} of $X\subset \R^d$ for a tree $T$ is
\begin{align*}
\Delta^T (X) &= \{ (|y^\tau-y^\sigma|)_{\tau\sim \sigma}: \tau, \sigma \in T,\ \{y^\tau\}_{\tau \in T} \text{ 
is a set of distinct points in } X\}\subset \R^{n-1},
\end{align*}
where $\vec{l}=(|y^\tau-y^\sigma|)_{\tau\sim \sigma}\in \R^{n-1}$ is an $(n-1)$-dimensional vector whose coordinates are distances between distinct points $y^\tau$ and $y^\sigma$, where $\tau \sim \sigma$. Generally, there is more than one way to label the vertices of $T$. Since relabeling vertices of $T$ correspond to permuting the coordinates of $\vec{l} \in \Delta^T (X)$ in a Euclidean space, the interior of $\Delta^T (X)$ and its Lebesgue measure are independent of the labeling of $T$. Thus, we canonically label vertices of depth $n$ using finite words of length $n+1$ with entries from $\Z_{\geq 0}$, where the root is always labeled by $0$. Specifically, given a labeling of vertices upto depth $d$, the vertices of depth $d+1$ is labeled by adjoining $i$ to the label of its unique predecessor if and only if it is $i$-th from the left. Then, the \textit{Kleene-Brouwer order} (for definition, see \cite{Moschovakis}) on the vertex set induces an order on the coordinates of the tree distances $\vec{l}$.

For example, a binary tree $T$ of depth 2 has the order (see Figure \ref{fig:kleene-brouwer})
\begin{align*}
0 > 00 > 01 > 000 > 001 > 010 > 011
\end{align*}
and determines the distance vector $\vec{l}=(l_1,\cdots, l_6)$ of distinct points $\{x^\tau \in X\}_{\tau \in T}$ to be
\begin{align*}
\vec{l}= \left( |x^0-x^{00}|, |x^0-x^{01}|, |x^{00}-x^{000}|, |x^{00}-x^{001}|, |x^{01}-x^{010}|, |x^{01}-x^{011}| \right).
\end{align*}
To this end, we define the \textit{pinned tree distance set} of $X\subset \R^d$ for a tree $T$ and pin $y\in X$ by
\begin{align*}
\Delta^T_y (X) &= \{ (|y^\tau-y^\sigma|)_{\tau\sim \sigma}: \tau, \sigma \in T,\ \{y^\tau\}_{\tau \in T} \text{ 
is a set of distinct points in } X, y^0=y\}.
\end{align*}
\begin{figure}
    \centering
    \includegraphics[width=0.4\linewidth]{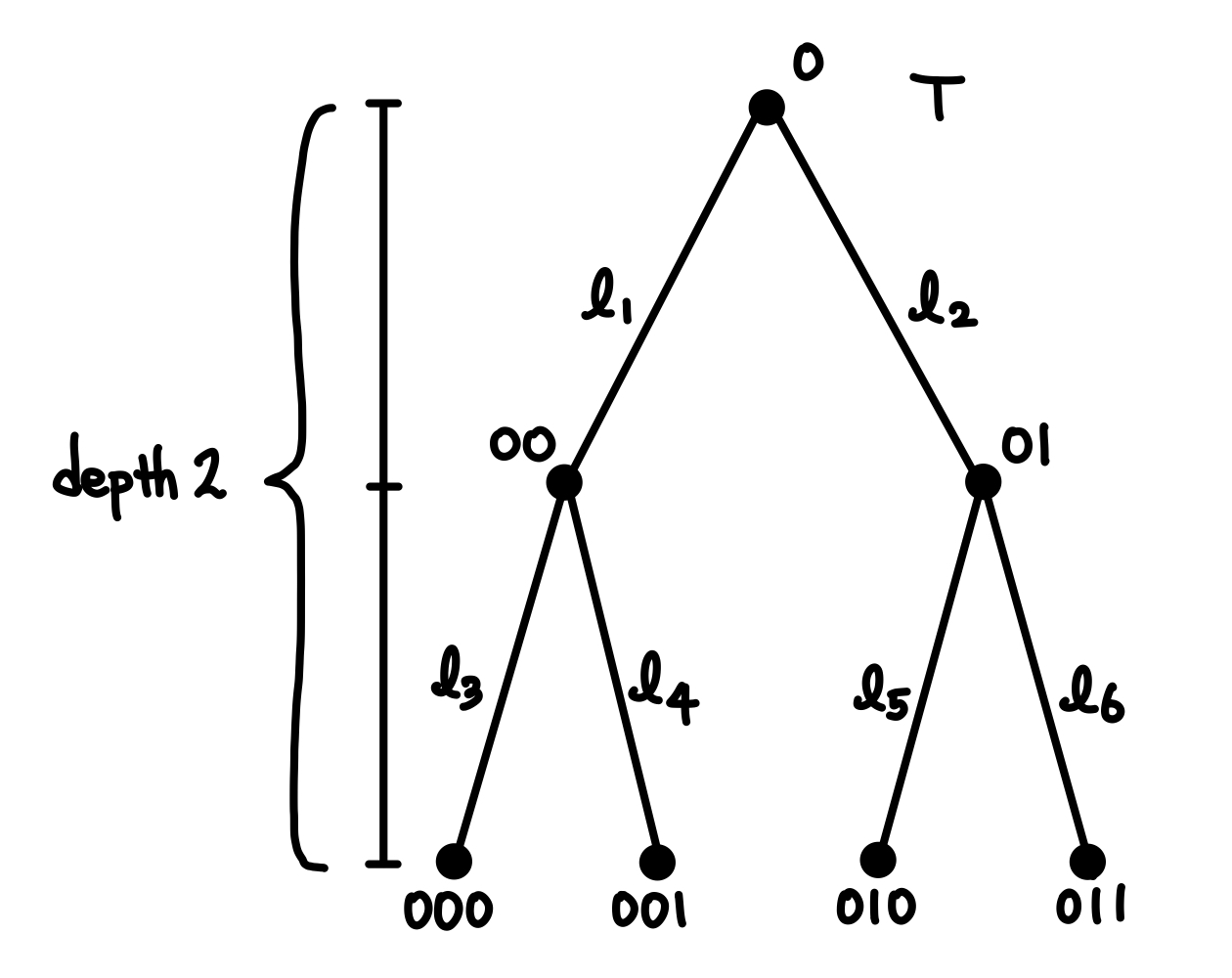}
    \caption{The canonical labeling for a binary tree $T$ of depth 2 with 7 vertices. The Kleene-Brouwer order on these vertices of $T$ induces the coordinate order of the distance vectors $\vec{l}$ in the tree distance set $\Delta^T(X)\subset \R^6$.}
    \label{fig:kleene-brouwer}
\end{figure}
We now state our main result in $\R^2$ as the following Theorem \ref{thm:tree}, which we prove in \S \ref{sec:tree}. Furthermore, the same proof in \ref{cor:2-chain-hausdorff-dimension-zero} gives a construction of a Cantor set in $\R^2$ of Hausdorff dimension 1 whose $T$ distance set has nonempty interior for a given finite tree $T$.

\begin{theorem}[Finite trees on subsets of $\R^2$]\label{thm:tree}\index{}
For each Cantor set $K \subset \mathbb{R}$ and finite tree $T$, there exists a Cantor set $\widetilde{K} \subset \R$ such that $\left( \Delta^T_{y} (K \times \widetilde{K} ) \right)^\circ \neq \varnothing$ for some $y\in K\times \widetilde{K}$.
\end{theorem}
\begin{corollary}[]\label{cor:tree-hausdorff-dimension-zero}\index{}
For any finite tree $T$, there exists a Cantor set $C\subset \R^2$ such that $\dim_{\rm H}(C)=1$ and the pinned tree distance set $\Delta^T_{y} (C)$ has nonempty interior for some $y\in C$.
\end{corollary}

\subsection{Generalization to $\R^d$ and to other types of distances}

In this section, we generalize previous results to $\R^d$ and to generalized distances. 
We first generalize distances to maps satisfying certain derivative conditions, following \cite{McDonald-Taylor-finite-point-configurations}.

\begin{definition}[]\label{def:derivative-condition}\index{}
Let $\phi:\R^d\times \R^d \to \R$, $X\subset \R^d$, and $T$ be a finite tree.
\begin{enumerate}
\item $\phi$ is said to satisfy the \textit{derivative condition} if there exists a nonempty open set $\mathcal{U}\subset \R^{d}$ such that for every pair of distinct points $a, b\in \mathcal{U}$, there exist nonempty open subsets $A, B \subset \mathcal{U}$ such that $a\in A$, $b\in B$, and
\begin{align}\label{eq:derivative-condition}
\frac{\partial \phi (x,y)}{\partial y_{i}} \neq 0 \quad \quad \forall (x,y)\in A\times B \quad \forall i\in \{1,\cdots, d\}.
\end{align}
\item If $\phi$ satisfies the derivative condition, we define:
\begin{enumerate}
\item \textit{$\phi$ distance set} as $\Delta^\phi (X) = \{ \phi (x,y): x, y\in X \}$;
\item \textit{pinned $\phi$ distance set} with pin $x\in X$ as $\Delta^\phi_x (X) = \phi (x, X) = \{ \phi (x,y): y\in X \}$;
\item \textit{$(\phi,T)$ distance set} as 
\begin{equation*}
\Delta^{\phi,T} (X)= \{ (\phi (x^\tau, x^\sigma) )_{\tau \sim \sigma}: \tau, \sigma \in T, \ \{x^\tau \in X\}_{\tau \in T} \text{ distinct } \}.
\end{equation*}
\item \textit{pinned $(\phi,T)$ distance set} with a pin $x\in X$ as
\begin{equation}\label{distance T phi set}
\Delta^{\phi,T}_x (X)= \{ (\phi (x^\tau, x^\sigma) )_{\tau \sim \sigma}: \tau, \sigma \in T, \ \{x^\tau \in X\}_{\tau \in T} \text{ distinct }, x^0=x \}.
\end{equation}
\end{enumerate}
\end{enumerate}
\end{definition}

\begin{example} We give several familiar examples as special cases of $(\phi, T)$ distance sets. Let $X\subset \R^{d}$ and $T$ be a finite tree.
    \begin{enumerate}
\item (Tree distance sets) For $X\subset \R^{d}$, the tree distance set $\Delta^T (X)$ is a special case of the $(\phi, T)$ distance set $\Delta^{\phi, T}(X)$ with $\phi:\R^{d}\times \R^{d}\to \R$ given by $\phi(x,y)= |x-y|$. Note that $\phi$ satisfies the derivative condition since \eqref{eq:derivative-condition} holds for $\mathcal{U}=\R^d$ and any disjoint open subsets $A, B \subset \mathcal{U}$.
\item ($p$-norm distance sets) The above example also applies with $\phi(x,y) = |x-y|_p$, where $|\cdot |_p$ is the $p$-norm in $\R^{d}$ with $p\geq 1$.
\item (Dot product sets) The \textit{tree dot product set} of a set $X\subset \R^d$ and tree $T$ with $n$ vertices is defined by
\begin{align*}
\{ (x^\tau \cdot x^\sigma)_{\tau\sim \sigma} : \{x^\tau\}_\tau  \text{ are distinct points in } X\}\subset \R^{n-1},
\end{align*}
where $x\cdot y = \sum_{i=1}^d x_i y_i$. If $T$ is a $1$-chain, then we recover the usual dot product set of $X\subset \R^{d}$ defined by $\{x\cdot y : x, y\in X\}$. By choosing $\mathcal{U}$ to be an open subset in $\R^{d}$ not containing $0$, we see that $\eqref{eq:derivative-condition}$ holds for all open subsets $A,B$ of $\mathcal{U}$.
\end{enumerate}
\end{example}

Furthermore, our key Lemma \ref{lem:pin-wiggling} in $\R$ for distance sets readily generalizes to $\R^d$ and for maps that satisfy the derivative conditions.
\begin{lemma}[Pin wiggling lemma in higher dimensions]\label{lem:pin-wiggling-Rd}
Let $\phi:\R^{2d}\times \R^{2d}\to \R$ be a map that satisfies the derivative condition, $K\subset \R^d$ be a Cantor set, $v=(v_1,v_2)\in \R^d\times \R^d$, and $Q\subset \text{conv}(K)$ be a nonempty open cube such that $K\cap Q\neq \varnothing$. Then, there exist a Cantor set $\widetilde{K}\subset Q$ with $|\widetilde{K}|>|K|$ and an open neighborhood $\mathcal{U}_v\subset \R^{2d}$ of $v$ such that
\begin{align*}
\left( \bigcap_{z\in \mathcal{U}_v} \phi (z, K\times \widetilde{K}) \right)^\circ \neq \varnothing.
\end{align*}
\end{lemma}

Using Lemma \ref{lem:pin-wiggling-Rd}, we obtain the most general result of this paper, Theorem \ref{thm:tree-Rd}, which we prove in \S \ref{sec:tree-Rd}. In particular, Theorem \ref{thm:2-chain}, Theorem \ref{thm:n-chain}, and Theorem \ref{thm:tree} are all special cases of Theorem \ref{thm:tree-Rd}. The same goes for our application of constructing Cantor sets with Hausdorff dimension $\frac{d}{2}$ and nonempty generalized distance sets.
\begin{theorem}[Finite trees on subsets of $\R^{2d}$]\label{thm:tree-Rd}\index{}
Let $\phi:\R^{2d}\times \R^{2d} \to \R$ be a map that satisfies the derivative condition in Definition \ref{def:derivative-condition} and let $K \subset \R^d$ be a Cantor set. Then, for every finite tree $T$, there exists a Cantor set $\widetilde{K} \subset \R^d$ with $|\widetilde{K}|>|K|$ such that $\left( \Delta^{\phi, T}_y (K \times \widetilde{K} ) \right)^\circ \neq \varnothing$ for some $y\in K\times \widetilde{K}$, where $\Delta^{\phi, T}_y(X)$ is the pinned tree distance set defined in \eqref{distance T phi set}.
\end{theorem}
\begin{corollary}[]\label{cor:tree-hausdorff-dimension-zero-Rd}\index{}
For any finite tree $T$, a map $\phi$ that satisfies the derivative condition, there exists a Cantor set $C\subset \R^{d}$ such that $\dim_{\rm H}(C)=\frac{d}{2}$ and the pinned $(\phi,T)$ distance set $\Delta^{\phi, T}_y (C)$ has nonempty interior for some $y\in C$.
\end{corollary}

\section{Proof of Lemma \ref{lem:pin-wiggling}}\label{sec:lem}
To prove Lemma \ref{lem:pin-wiggling}, we use the following result in \cite[Theorem 1.1]{jung-lai-interior-of-certain}. It is interesting to note that Theorem \ref{thm:jung-lai} does not rely on the Newhouse gap lemma, but rather on what the authors call the containment lemma. 

\begin{theorem*}{A}[Jung-Lai]\label{thm:jung-lai}
    Let $N\ge 1$ and $\Lambda\subset \R^N$ be a set with interior and let $\alpha_0\in \Lambda^{\circ}$. Let $Q_1 \subset Q_2\subset \R^d$ be open cubes and $Q_1 \ne \varnothing$.  Let $H:\Lambda\times Q_1\times Q_2 \to \R^d$ be a $C^1$ function such that the Jacobian on  $Q_2$ is invertible. Then for all Cantor sets $K_1\subset Q_1$, there exists a Cantor set $K_2\subset Q_2$ with $|K_2|>|K_1|$ and an open cube centered at $\alpha_0$, ${\mathtt Q}_{\epsilon}(\alpha_0) \subset \Lambda$ ($\epsilon>0$) such that 
    $$
    \left(\bigcap_{\alpha\in {\mathtt Q}_{{\epsilon}}(\alpha_0)} H(\alpha, K_1,K_2)\right)^{\circ}\ne \varnothing.
    $$
\end{theorem*}
Recall that if $\Lambda \subset \R^N$, $Q_1, Q_2 \subset \R^d$, and $H:\Lambda \times Q_1 \times Q_2 \to \R^d$ is a $C^1$ function, then denoting  $H(\lambda, x, y)=(H_1 (\lambda, x, y), \cdots, H_d (\lambda, x, y))$, the \textit{Jacobian} of $H$ on $Q_2$ is defined as the following $d\times d$ matrix:
\begin{align*}
    \left(\frac{\partial H_i(\lambda, x, y)}{\partial y_j}\right)_{i,j=1}^d.
\end{align*}

\begin{proof}[Proof of Lemma \ref{lem:pin-wiggling}]
We first prove a claim.

\noindent \textit{Claim}. There exist a nonempty open interval $Q\subset I$ and an open neighborhood $\Lambda_2$ of $v_2$ in $\R$ disjoint from $Q$ such that $K \cap Q \neq \varnothing$ contains a sub-Cantor set of $K$.

\begin{proof}[Proof of the claim]
This is straightforward if $v_2\not\in I$. If $v_2 \in I$, then let $p\in (K\cap I)\setminus \{v_2\}$. Then, there exist disjoint open intervals $Q, \Lambda_2 \subset I$ that separate $p$ and $v_2$, i.e., $p\in Q$ and $v_2\in \Lambda_2$. Since $Q$ contains a closed interval that contains $p$, $K\cap Q$ must contain a sub-Cantor set of $K$. This finishes the proof of the claim.
\end{proof}

Returning to the main proof, we are now ready to use Theorem \ref{thm:jung-lai}. Let $Q_1=Q_2=Q$ and $\Lambda = \R\times \Lambda_2$. Define
\begin{align*}
    H:\Lambda\times \R \times \R \to \R : (z,x,y)\mapsto \|(x,y)-z\|_2,
\end{align*}
where $z=(z_1,z_2)\in \R^2$ and $\|\cdot \|_2$ is the Euclidean norm in $\R^2$. By (i) of the above claim, we have
\begin{align*}
    \left|\frac{\partial H (z, x, y)}{\partial y}\right| = \frac{|y-z_2|}{\|(x,y)-z\|_2}>0 \quad \forall (z, x, y)\in \Lambda\times Q_1 \times Q_2.
\end{align*}
Hence, the Jacobian of $H$ on $I$ is invertible and the conditions of Theorem \ref{thm:jung-lai} are satisfied. By (ii) of the above claim, there exists a Cantor set $K_1\subset K\cap Q_1$. By Theorem \ref{thm:jung-lai}, there exist a Cantor set $K_2 \subset Q_2$ with $|K_2|>|K_1|$ and an open neighborhood $U_v\subset \Lambda$ of $v$ such that
\begin{align*}
     \left(\bigcap_{z\in U_v} H(z, K_1, K_2)\right)^\circ \neq \varnothing.
\end{align*}
Take $\widetilde{K}=K_2$ as our desired Cantor set. Then, since $K_1 \subset K$, we have
\begin{align*}
    \left(\bigcap_{z\in U_v} \Delta_z(K\times \widetilde{K})\right)^\circ=\left(\bigcap_{z\in U_v} H(z, K, \widetilde{K})\right)^\circ \supset \left(\bigcap_{z\in U_v} H(z, K_1, K_2)\right)^\circ \neq \varnothing.
\end{align*}
This completes the proof.
\end{proof}

\section{Proof of Theorem \ref{thm:2-chain}}\label{sec:2-chain}
\begin{proof}[Proof of Theorem \ref{thm:2-chain}]
Let $x^0, x^1, x^2 \in K \times K $ be three distinct points 
on the diagonal line $\{y=x\}$. Let $\mathcal{U}_2$ be an open neighborhood of $x^2$ in $\R^2$ such that $x^0, x^1\not\in \mathcal{U}_2$.
Since $x^2\in K\times K$ is in the diagonal line, there exists a nonempty open interval $I\subset \R$ such that $x^2 \in I\times I\subset \mathcal{U}_2$.
Applying Lemma \ref{lem:pin-wiggling} to $K_2=K\cap I$, $x^1$, and $I$, there exist a Cantor set $\widetilde{K}_2\subset I$ and an open neighborhood $\mathcal{U}_1$ of $x^1$ in $\R^2$ such that $K_2\times \widetilde{K}_2 \subset \mathcal{U}_2$ and
\begin{align}\label{eq:2-chain-interior-1}
    \left(\bigcap_{x\in \mathcal{U}_1} \Delta_x (K_2 \times \widetilde{K}_2) \right)^\circ \neq \varnothing.
\end{align}
Note that we may assume $\mathcal{U}_1$ does not contain $x^0$ by replacing it with a suitable open subset. Repeating the same argument with $\mathcal{U}_1$ in place of $\mathcal{U}_2$, $x^1$ in place of $x^2$, and $x^0$ in place of $x^1$, there exist a sub-Cantor set $K_1\subset K$, a Cantor set $\widetilde{K}_1$, and an open neighborhood $\mathcal{U}_0$ of $x^0$ in $\R^2$ such that $K_1\times \widetilde{K}_1 \subset \mathcal{U}_1$ and
\begin{align}\label{eq:2-chain-interior-2}
    \left(\bigcap_{x\in \mathcal{U}_0} \Delta_x (K_1 \times \widetilde{K}_1)\right)^\circ \neq \varnothing.
\end{align}
Let $\widetilde{K} = K\cup \left(\bigcup_{i=1}^2 \widetilde{K}_i\right)$ and $y=y^0 \in \mathcal{U}_0 \cap (K\times \widetilde{K})$, which is nonempty since it contains $x^0$. To prove that $\left(\Delta^2_{y} (K\times \widetilde{K})\right)^\circ \neq \varnothing$, by \eqref{eq:2-chain-interior-1} and \eqref{eq:2-chain-interior-2}, it suffices to show that
\begin{align}
    \Delta^2_{y} (K\times \widetilde{K}) \supset \left( \bigcap_{x\in \mathcal{U}_0} \Delta_x (K_1 \times \widetilde{K}_1) \right)^\circ \times \left( \bigcap_{x\in \mathcal{U}_1} \Delta_x (K_2 \times \widetilde{K}_2) \right)^\circ. \label{eq:2-chain-3}
\end{align}
To see that \eqref{eq:2-chain-3} holds, let
\begin{align*}
    (l^1,l^2) \in \left( \bigcap_{x\in \mathcal{U}_0 } \Delta_x (K_1 \times \widetilde{K}_1) \right)^\circ \times \left( \bigcap_{x\in \mathcal{U}_1} \Delta_x (K_2 \times \widetilde{K}_2) \right)^\circ.
\end{align*}
Then, since $l^1\in \left( \bigcap_{x\in \mathcal{U}_0} \Delta_x (K_1 \times \widetilde{K}_1) \right)^\circ$ and $y^0\in \mathcal{U}_0$, there exists $y^1\in K_1 \times \widetilde{K}_1\subset \mathcal{U}_1$ such that $l^1=|y^0-y^1|$. Also, since $l^2\in \left( \bigcap_{x\in \mathcal{U}_1} \Delta_x (K_2 \times \widetilde{K}_2) \right)^\circ$ and $y^1\in K_1 \times \widetilde{K}_1\subset \mathcal{U}_1$, there exists $y^2\in K_2 \times \widetilde{K}_2 \subset K\times \widetilde{K}$ such that $l^2=|y^1-y^2|$. Thus,
\begin{align*}
(l^1,l^2)\in \{ (|y^0-y^1|,|y^1-y^2|): y^1,y^2\in K\times \widetilde{K} \}= \Delta_y^{(2)} (K\times \widetilde{K}).
\end{align*}
This shows \eqref{eq:2-chain-3} and the proof is complete.
\end{proof}

\section{Proof of Theorem \ref{thm:n-chain}}\label{sec:n-chain}
\begin{proof}[Proof of Theorem \ref{thm:n-chain}]
Let $x^0, x^1, x^2, \cdots, x^{n-1} \in K \times K $ be $n$ distinct points lying on the diagonal line. We first prove the following claim, which is illustrated in Figure \ref{fig:5-chain}.

\noindent\textit{Claim.} There exist sub-Cantor sets $K_i$ of $K$, Cantor sets $\widetilde{K}_i\subset \R$ for each $i\in \{1, 2, \cdots, n-1\}$ and pairwise disjoint open neighborhoods $\mathcal{U}_j$ of $x^j$ for each $j\in \{0,1,\cdots, n-1\}$ such that for each $i\in \{1, 2, \cdots, n-1\}$,
\begin{enumerate}[(i)]
\item $K_i\times \widetilde{K}_i \subset \mathcal{U}_{i}$;
\item $\left(\bigcap_{x\in \mathcal{U}_{i-1}} \Delta_x (K_{i}\times \widetilde{K}_{i})\right)^\circ \neq \varnothing$;
\end{enumerate}
\begin{proof}[Proof of the claim]
We proceed inductively in reverse order. Let $\mathcal{U}_{n-1}$ be an open neighborhood of $x^{n-1}$ in $\R^2$ such that $x^0,\cdots, x^{n-2} \not\in \mathcal{U}_{n-1}$. Since $x^{n-1}\in K\times K$ lies in the diagonal line, there exists a nonempty open interval $I_{n-1}\subset \R$ such that $x^{n-1}\in I_{n-1}\times I_{n-1} \subset \mathcal{U}_{0i}$.
Then, applying Lemma \ref{lem:pin-wiggling} to $K_{n-1}=K\cap I_{n-1}$, $x^{n-2}$, and $I_{n-1}$, there exist a Cantor set $\widetilde{K}_{n-1} \subset I_{n-1}$ and an open neighborhood $\mathcal{U}_{n-2}$ of $x^{n-2}$ in $\R^2$ such that
\begin{align*}
K_{n-1}\times \widetilde{K}_{n-1}\subset \mathcal{U}_{n-1} \quad \text{and} \quad \left(\bigcap_{x\in \mathcal{U}_{n-2}} \Delta_x (K_{n-1}\times \widetilde{K}_{n-1})\right)^\circ \neq \varnothing.
\end{align*}
Suppose that, for some $j\in\N$, we have obtained the desired sets
\begin{align*}
\{\mathcal{U}_j\} \cup \{(K_i, \widetilde{K}_{i}, \mathcal{U}_i): i\in \{j+1,j+2, \cdots, n-1\}\},
\end{align*}
such that (i)-(ii) hold for $i\in\{j+1, j+2, \cdots, n-1\}$. Similarly as before, since $x^{j}\in K\times K$ is a point on the diagonal line, there exists a nonempty open interval $I_j\subset \R$ such that $x^{j}\in I_{j}\times I_j \subset \mathcal{U}_{j}$. Thus, using Lemma \ref{lem:pin-wiggling} on $K_{j}=K\cap I_j$, $x^{j-1}$, and $I_j$, we deduce that there exist a Cantor set $\widetilde{K}_j\subset I_j$ and an open neighborhood $\mathcal{U}_{j-1}$ of $x^{j-1}$ in $\R^2$ such that (i)-(ii) hold for $i\in\{j, j+1, \cdots, n-1\}$. Note that we may restrict $\mathcal{U}_{j-1}$ to its subset so that $\mathcal{U}_{j-1}, \cdots, \mathcal{U}_{n-1}$ are pairwise disjoint. Proceeding this process down to $j=1$, we obtain the desired sub-Cantor sets $K_i \subset K$, Cantor sets $\widetilde{K}_i\subset \R$, and open neighborhoods $\mathcal{U}_j$ of $x^j$ for each $i\in \{1,2, \cdots, n-1\}$ and $j\in \{0,1,\cdots, n-1\}$ that satisfy (i)-(ii).
\end{proof}

We now return to the main proof. Define $\widetilde{K} = K\cup \left(\bigcup_{i=1}^{n-1} \widetilde{K}_i\right)$ and let $y=y^0\in \mathcal{U}_0 \cap (K\times \widetilde{K})$, which is nonempty since it contains $x^0$. By (ii) of the above claim, to prove that $\Delta^n_{y} (K\times \widetilde{K})$ has nonempty interior, it suffices to show that
\begin{align}
    \Delta^n_{y} (K\times \widetilde{K}) \supset \left( \bigcap_{x\in \mathcal{U}_0} \Delta_x (K_1 \times  \widetilde{K}_1) \right)^\circ \times \cdots  \times \left( \bigcap_{x\in \mathcal{U}_{n-2}} \Delta_x (K_{n-1} \times \widetilde{K}_{n-1}) \right)^\circ. \label{eq:inclusion}
\end{align}
To prove \eqref{eq:inclusion}, let
\begin{align*}
    (l^0, \cdots, l^{n-2}) &\in \left( \bigcap_{x\in \mathcal{U}_0} \Delta_x (K_1 \times  \widetilde{K}_1) \right)^\circ \cdots \times \left( \bigcap_{x\in \mathcal{U}_{n-2}} \Delta_x (K_{n-1} \times \widetilde{K}_{n-1}) \right)^\circ.
\end{align*}
\eqref{eq:inclusion} holds if we find distinct points $y^1, y^2, \cdots, y^{n-1} \in K\times \widetilde{K}$ (different from $y^0$) such that
\begin{equation}\label{eq:1}
    l^{i+1}=|y^i-y^{i+1}| \quad (i\in \{0,\cdots, n-2\}).
\end{equation}
Each $y^i$ will be in $\mathcal{U}_i$ and we use the claim above to find these points inductively. Since $l^0\in \left( \bigcap_{x\in \mathcal{U}_0} \Delta_x (K_1 \times \widetilde{K}_1) \right)^\circ$ and $y^0\in \mathcal{U}_0$, there exists $y^1\in K_1 \times \widetilde{K}_1$ such that $l^0=|y^0-y^1|$. Next, since $l^1 \in \left( \bigcap_{x\in \mathcal{U}_1} \Delta_x (K_2 \times \widetilde{K}_2) \right)^\circ$ and $y^1\in \mathcal{U}_1$ by (i), there exists $y^2\in K_2 \times \widetilde{K}_2$ such that $l^1=|y^1-y^2|$. Note that $y^1, y^2$ are distinct since $\mathcal{U}_0$ and $\mathcal{U}_1$ are disjoint. Suppose that for some $j\in \N$, we have found distinct points $y^0, \cdots, y^j\in K\times \widetilde{K}$ such that $y^i \in \mathcal{U}_i$ and $l^i=|y^i-y^{i+1}|$ for each $i=0,1,\cdots, j$. Then, since $l^{j+1} \in \left(\bigcap_{x\in \mathcal{U}_{j}} \Delta_x (K_{j+1}\times \widetilde{K}_{j+1})\right)^\circ$, there exists $y^{j+1}\in K_{j+1}\times \widetilde{K}_{j+1}$ such that $l^{j+1} =|y^j-y^{j+1}|$. By (i), $y^{j+1}\in \mathcal{U}_{j+1}$, and since $\{\mathcal{U}_i: i\in \{0, 1, \cdots, j+1\}\}$ are pairwise disjoint, $y^0,\cdots, y^{j+1}$ are distinct. Proceeding upto $j=n-2$, we obtain all the desired distinct points $y^0, y^1,\cdots, y^{n-1}\in K\times \widetilde{K}$ such that \eqref{eq:1} holds. This proves \eqref{eq:inclusion}, and the proof is complete.
\end{proof}
\begin{figure}
\centering
\includegraphics[scale=0.33]{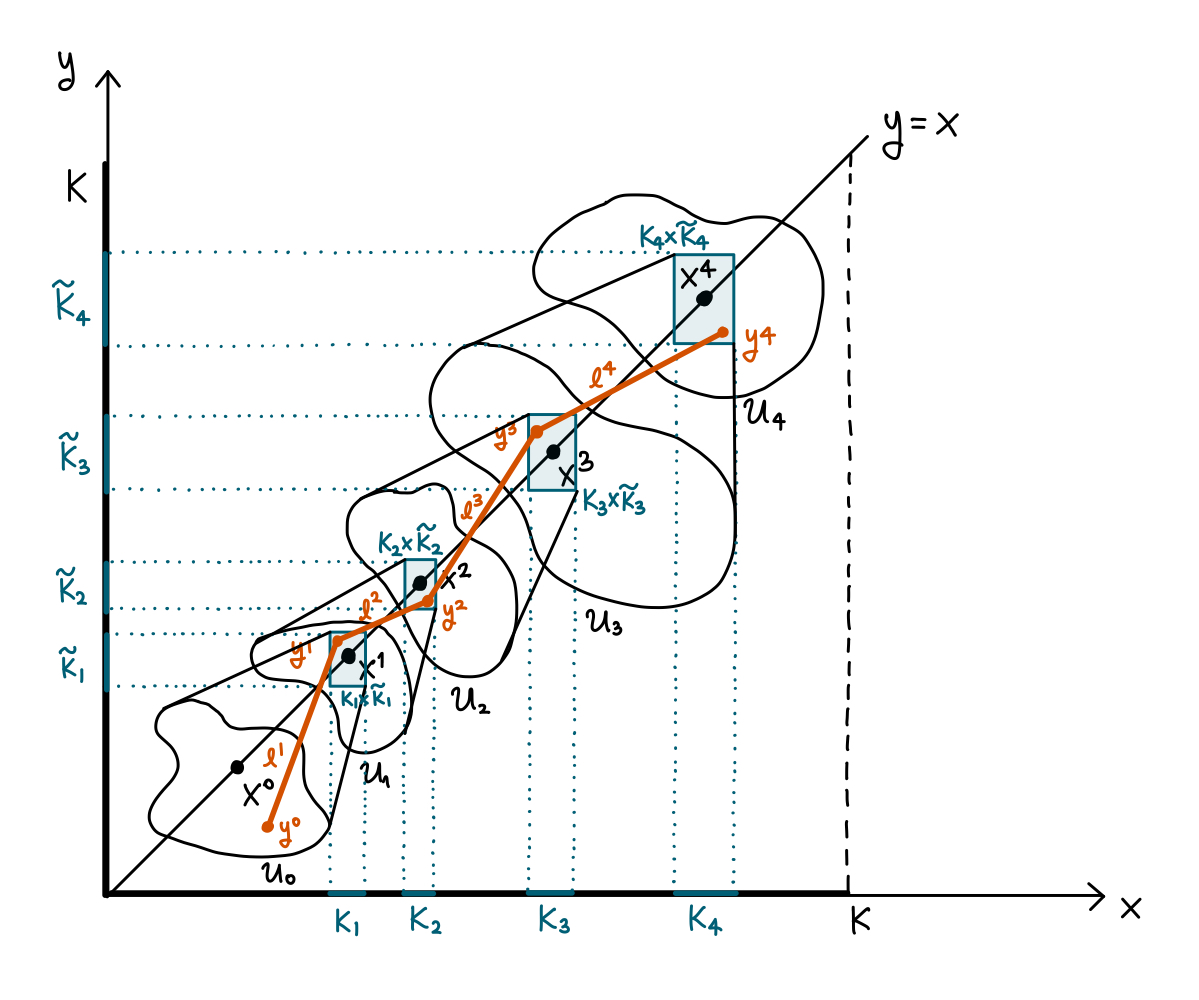}
\caption{An example of $5$-chain in the proof of Theorem \ref{thm:n-chain} in \S \ref{sec:n-chain}.}
\label{fig:5-chain}
\end{figure}

\section{Proof of Theorem \ref{thm:tree}}\label{sec:tree}
\begin{proof}[Proof of Theorem \ref{thm:tree}]
We first need the following claim, which is illustrated in figure \ref{fig:5-tree}.

\noindent\textit{Claim.} For every finite tree $T$ with root $0$ and distinct points $\{x^\tau \in K\times K: \tau \in T\}$, there exist sub-Cantor sets $K_\tau$ of $K$, Cantor sets $\widetilde{K}_\tau \subset \R$ for each $\tau \in T \setminus \{0\}$ and pairwise disjoint open neighborhoods $\mathcal{U}_\sigma$ of $x^\sigma$ for each $\sigma \in T$ such that for each $\tau \in T\setminus \{0\}$,
\begin{enumerate}[(i)]
\item $K_\tau\times \widetilde{K}_\tau \subset \mathcal{U}_{\tau}$;
\item $\left(\bigcap_{x\in \mathcal{U}_{\tau'}} \Delta_x (K_{\tau}\times \widetilde{K}_{\tau})\right)^\circ \neq \varnothing$, where $\tau'$ is the unique parent of $\tau$.
\end{enumerate}

\begin{proof}[Proof of the claim.]
We prove by induction on the depth of trees. For the base case, let $T$ be a tree of depth 1 with $n\geq 2$ vertices $0, 00, 01, \cdots, 0(n-2)$. Let $\{x^\tau \in K\times K: \tau \in T\}$ be distinct points and $\{\mathcal{U}_{0i}: i\in \{0, 1, \cdots, n-2\}\}$ be pairwise disjoint open sets in $\R^2$, where $\mathcal{U}_{0i}$ is an open neighborhoods of $x^{0i}$ in $\R^2$ and $x^0, x^{0j}\not\in \mathcal{U}_{0i}$ for each $j\neq i$. Note that since $x^{0i}\in K\times K$ lies in the diagonal line, there exists a nonempty open interval $I_i\subset \R$ such that $x^{0i}\in I_i\times I_i \subset \mathcal{U}_{0i}$. Then, for each $i\in \{0,\cdots, n-2\}$, applying Lemma \ref{lem:pin-wiggling} to $K_{0i} = K\cap I_{0i}$, $x^{0}$, and $I_i$, we deduce that there exist a Cantor set $\widetilde{K}_{0i} \subset I_i$ and an open neighborhood $\mathcal{V}_{0}^i$ of $x^0$ in $\R^2$ such that $K_{0i}\times \widetilde{K}_{0i}\subset \mathcal{U}_{0i}$ and
\begin{align*}
\left(\bigcap_{x\in \mathcal{V}_{0}^i} \Delta_x (K_{0i}\times \widetilde{K}_{0i})\right)^\circ \neq \varnothing.
\end{align*}
Define $\mathcal{U}_0 = \bigcap_{i=0}^{n-2} \mathcal{V}_0^i$, which is an open neighborhood of $x^0$ in $\R^2$, which we may assume, by taking its suitable open subset, to be disjoint from $U_{0i}$ for all $i\in\{0,\cdots, n-2\}$. Then, (i)-(ii) holds for the tree $T$ of depth 1.

For the induction step, suppose that the claim holds for all trees of depth at most $d$. Let $T$ be a tree of depth $d+1$ with root $0$ and $\{x^\tau \in K\times K: \tau \in T\}$ be distinct points. Denote $T(k)=\{\sigma \in T: \sigma \text{ has depth } k\}$ to be the vertices of $T$ of depth $k$ and for each $\sigma \in T$, denote $T_\sigma=\{\tau\in T: \sigma<\tau\}$ to be the subtree of $T$ that has $\sigma$ as its root and includes all of its vertices below it. Since the tree $T$ with the root $0$ removed is a union of disconnected trees $\{T_\sigma: \sigma \in  T (1)\}$ of depth at most $d$, by the induction hypothesis, for each $\sigma\in T (1)$, we have all the desired sub-Cantor sets $K_\tau$ of $T$ and Cantor sets $\widetilde{K}_\tau \subset \R$ except for $\tau\in T(1)$ and all the desired pairwise disjoint open neighborhoods $\mathcal{U}_\sigma$ of $x^\sigma$ except for $\sigma=0$. Similar to above, for each $\sigma \in  T (1)$, since $x_\sigma$ lies in the diagonal line, Lemma \ref{lem:pin-wiggling} states that there exist a sub-Cantor set $K_\sigma\subset K$, a Cantor set $\widetilde{K}_\sigma$, and an open neighborhood $\mathcal{V}_0^\sigma$ of $x^0$ in $\R^2$ such that
\begin{align*}
    K_\sigma \times \widetilde{K}_\sigma \subset \mathcal{U}_\sigma \quad \text{and} \quad \left(\bigcap_{x\in \mathcal{V}_{0}^\sigma} \Delta_x (K_{0i}\times \widetilde{K}_{0i})\right)^\circ \neq \varnothing.
\end{align*}

Define $\mathcal{U}_0=\bigcap_{\sigma\in T (1)} \mathcal{V}_0^\sigma$, which is an open neighborhood of $x^0$ in $\R^2$, and note that we may assume it to be disjoint from $\mathcal{U}_\tau$ for each $\tau \neq 0$ by taking a suitable open subset. This shows that the claim holds for all trees of depth $d+1$, and hence, by induction on the height of trees, the claim holds for all finite trees $T$.
\end{proof}

We now return to the main proof. Let $T$ be a finite tree of depth $d$ and $\{x^\tau \in K\times K: \tau \in T\}$ be distinct points. Using the claim above, define $\widetilde{K} = K\cup \bigcup_{\tau\in T\setminus \{0\}} \widetilde{K}_\tau$. Let $y=y^0 \in \mathcal{U}_0 \cap (K\times \widetilde{K})$. By (ii), to prove that $\left(\Delta^T (K\times \widetilde{K})\right)^\circ \neq \varnothing$, it suffices to show
\begin{align}
    \Delta^T (K\times \widetilde{K}) &\supset \underbrace{\left( \bigcap_{x\in \mathcal{U}_{{\sigma_{1,0}}'}} \Delta_x (K_{\sigma_{1,0}} \times  \widetilde{K}_{\sigma_{1,0}}) \right)^\circ \times \cdots \times \left( \bigcap_{x\in \mathcal{U}_{{\sigma_{1,n_1}}'}} \Delta_x (K_{\sigma_{1,n_1}} \times  \widetilde{K}_{\sigma_{1,n_1}}) \right)^\circ}_{ T (1)=\{\sigma_{1,0}, \cdots, \sigma_{1, n_1} \}}\times \nonumber\\
    \ & \vdots \nonumber \\
    &\times \underbrace{\left( \bigcap_{x\in \mathcal{U}_{{\sigma_{d,0}}'}} \Delta_x (K_{\sigma_{d,0}} \times  \widetilde{K}_{\sigma_{d,0}}) \right)^\circ \times \cdots \times \left( \bigcap_{x\in \mathcal{U}_{{\sigma_{d,n_d}}'}} \Delta_x (K_{\sigma_{d,n_d}} \times  \widetilde{K}_{\sigma_{d,n_d}}) \right)^\circ}_{ T (d)=\{\sigma_{d,0}, \cdots, \sigma_{d,n_d}\}},\label{eq:n-tree-inclusion}
\end{align}
where $\sigma'$ is the unique parent of $\sigma$ in tree $T$. To prove \eqref{eq:n-tree-inclusion}, let
\begin{align*}
    \vec{l}_T=\left(\underbrace{l^{\sigma_{1,0}}, \cdots, l^{\sigma_{1,n_1}}}_{T (1)=\{\sigma_{1,0}, \cdots, \sigma_{1,n_1}\}}, \cdots, \underbrace{l^{\sigma_{1,0}}, \cdots, l^{\sigma_{n_{1,0}}}}_{ T (d)=\{\sigma_{d,0}, \cdots, \sigma_{d,n_d}\}}\right)
\end{align*}
be an element on the right hand side of \eqref{eq:n-tree-inclusion}. Then, proving \eqref{eq:n-tree-inclusion} is equivalent to finding distinct points $\{y^\tau\}_{\tau\in T\setminus \{0\}}$ of $K\times \widetilde{K}$ such that
\begin{equation}\label{eq:n-tree-2} 
    \vec{l}_T= (|y^\tau - y^\sigma|)_{\tau \sim \sigma}.
\end{equation}
We again proceed inductively on the depth of trees. For the base case, suppose that $T$ is a tree of depth 1. Then, for each $\sigma\in  T (1)$, since $y^0\in \mathcal{U}_0$ and since $l^{\sigma} \in \bigcap_{x\in \mathcal{U}_{0}} \Delta_x (K_{\sigma}\times K_{\sigma})$ by (iv), there exists $y^\sigma \in K_\sigma\times K_{\sigma}$ such that $l^{\sigma} = |y^0-y^\sigma|$. Thus, \eqref{eq:n-tree-2} holds for the case of depth 1. For the induction hypothesis, suppose that for every tree of depth at most $d$, there exists distinct points $\{y^\tau\}_{\tau \in T\setminus \{0\}}$ such that \eqref{eq:n-tree-2} holds. Let $T$ be a tree of depth $d+1$. Then, for each $\sigma\in  T (1)$, since $y^0\in \mathcal{U}_0$ and since $l^{\sigma} \in \bigcap_{x\in \mathcal{U}_{0}} \Delta_x (K_{\sigma}\times K_{\sigma})$ by (ii), there exists $y^\sigma \in K_\sigma\times K_{\sigma}$ such that $l^{\sigma} = |y^0-y^\sigma|$. Moreover, since $T$ with the root $0$ removed is a union of disconnected trees $\{T_\sigma: \sigma \in  T (1)\}$ of depth at most $d$, by the induction hypothesis, each $T_\sigma$ has distinct points $\{y^{\tau}\}_{\tau \in T(\sigma)\setminus \{\sigma\} }$ such that \eqref{eq:n-tree-2} holds. Thus, we have found the desired distinct points $\{y^\tau\}_{\tau \in T\setminus \{0\} }$ such that \eqref{eq:n-tree-2} holds. By induction on the depth of trees, \eqref{eq:n-tree-2} holds for all finite trees. This completes the proof.
\end{proof}

\begin{figure}
\centering
\includegraphics[scale=0.255]{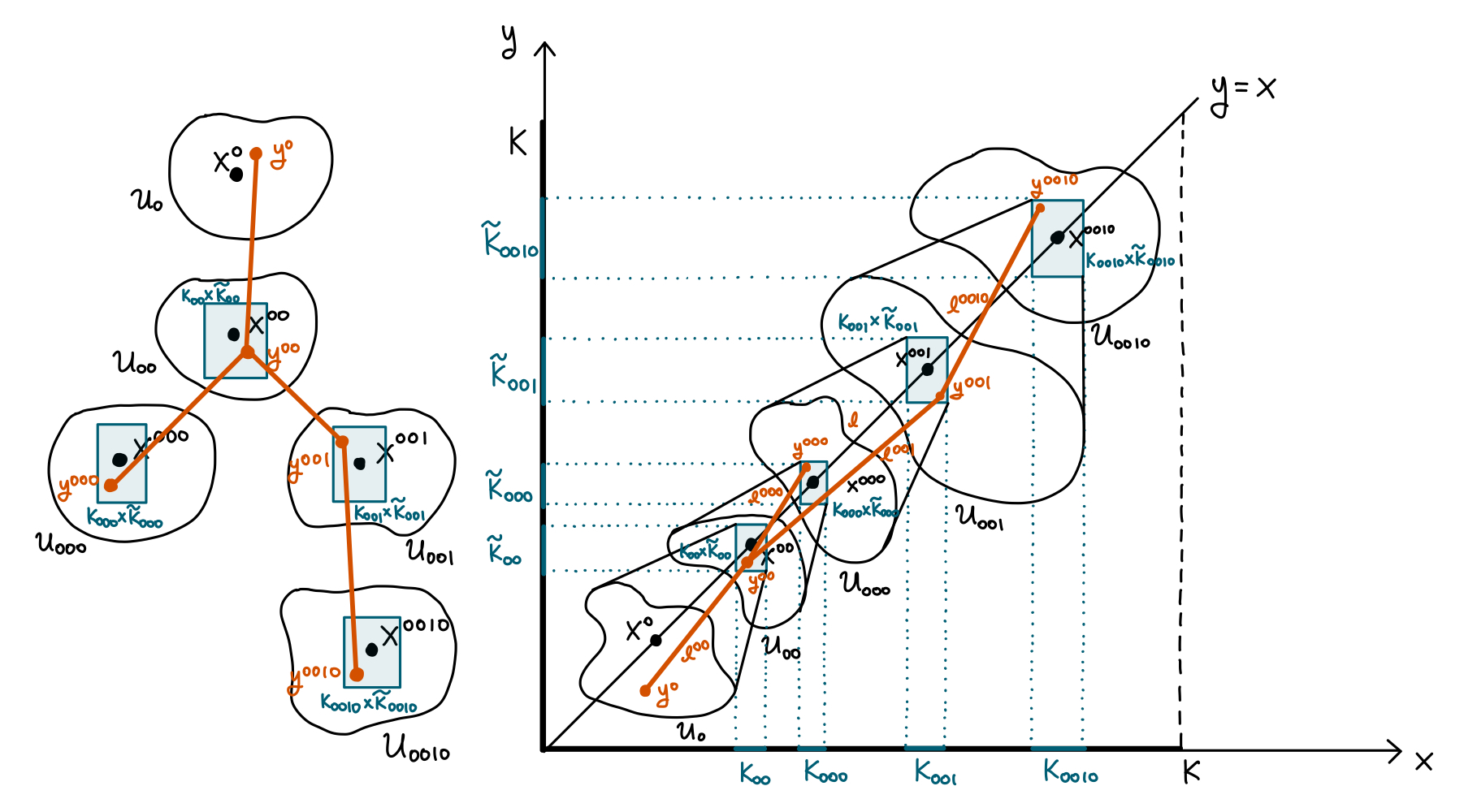}
\caption{An example of $5$-tree in the proof of Theorem \ref{thm:tree} in \S \ref{sec:tree}.}
\label{fig:5-tree}
\end{figure}

\section{Proof of Lemma \ref{lem:pin-wiggling-Rd}}\label{sec:lem-Rd}
\begin{proof}[Proof of Lemma \ref{lem:pin-wiggling-Rd}]
We use Theorem \ref{thm:jung-lai} analogously to the proof of Lemma \ref{lem:pin-wiggling}. Define
\begin{align*}
H:\R^{2d} \times \R^d \times \R^d \to \R^d : (z,x,y)\mapsto \left( \phi(z,(x,y)) , y_2, \cdots, y_d\right).
\end{align*}
The Jacobian determinant of $H$ with respect to $y$, whenever defined, is given by
\begin{align*}
\det \begin{pmatrix}
    \frac{\partial \phi(z, (x,y))}{\partial y_1} & \frac{\partial \phi(z, (x,y))}{\partial y_2} & \frac{\partial \phi(z, (x,y))}{\partial y_3} & \cdots & \frac{\partial \phi(z, (x,y))}{\partial y_d} \\
    0 & 1 & 0 & \cdots & 0 \\
    0 & 0 & 1 & \cdots & 0 \\
    \vdots & \vdots & \vdots & \ddots & 0 \\
    0 & 0 & 0 & \cdots & 1
\end{pmatrix}
=\frac{\partial \phi(z, (x,y))}{\partial y_1}.
\end{align*}
Let $k\in K\cap Q$. Since $\phi$ satisfies the derivative condition, there exists an open cube $\Lambda \subset \R^{2d}$ containing $v$ and an open cube $Q_2 \subset Q$ such that $k\in Q_2$ and
\begin{align*}
\frac{\partial \phi(z, (x,y))}{\partial y_1} \neq 0 \quad \forall z \in A \text{ and } \forall (x,y) \in Q_2 \times Q_2.
\end{align*}
This implies that the Jacobian of $H$ with respect to $y$ on $Q_2$ is invertible and the conditions of Theorem \ref{thm:jung-lai} are satisfied. For notational coherence, denote $Q_1=Q_2$. Then, since $k\in K\cap Q_1\neq \varnothing$, there exists a Cantor set $K_1 \subset K\cap Q_1$. By Theorem \ref{thm:jung-lai}, there exist a Cantor set $K_2\subset Q_2$ with $|K_2|>|K_1|$ and an open neighborhood $\mathcal{U}_v\subset \Lambda$ of $v$ such that
\begin{align*}
     \left(\bigcap_{z\in U_v} H(z, K_1, K_2)\right)^\circ \neq \varnothing.
\end{align*}
Taking $\widetilde{K}=K_2$ as our desired Cantor set, since $K_1 \subset K$, we have
\begin{align*}
    \left(\bigcap_{z\in U_v} \phi(z, K\times \widetilde{K})\right)^\circ=  \left(\bigcap_{z\in U_v} \pi_1 (H(z, K, \widetilde{K}))\right)^\circ \supset \pi_1 \left(\bigcap_{z\in U_v} H(z, K_1, K_2)\right)^\circ \neq \varnothing,
\end{align*}
where $\pi_1:\R^d\to \R$ is the projection map onto the first coordinate. Note that the last inclusion follows since projection is an open map and projection of an intersection is a subset of the intersection of projections. This completes the proof.
\end{proof}

\section{Proof of Theorem \ref{thm:tree-Rd}}\label{sec:tree-Rd}
\begin{proof}[Proof of Theorem \ref{thm:tree-Rd}]
The proof is a direct generalization of Theorem \ref{thm:tree}. Recall the claim in the proof of Theorem \ref{thm:tree} with the following (ii)$'$ in place of (ii) for each $\tau\in T\setminus \{0\}$:

\noindent (ii)$'$ $\left( \bigcap_{x\in \mathcal{U}_{\tau '}} \phi(x, K_\tau \times \widetilde{K}_\tau) \right)^\circ \neq \varnothing$, where $\tau'$ is the unique parent of $\tau$.

This version of the claim follows by using induction on the depth of trees in the proof of the original claim in Theorem \ref{thm:tree} with Lemma \ref{lem:pin-wiggling-Rd} in place of Lemma \ref{lem:pin-wiggling}. Returning to the main proof, let $T$ be a finite tree of depth $d$ and let $\{x^\tau \in K\times K: \tau \in T\}$ be distinct points. Using the claim above, define $\widetilde{K} = K\cup \  \bigcup_{\tau \in T\setminus \{0\}} \widetilde{K}_\tau$ and let $y\in \mathcal{U}_0 \cap (K\times \widetilde{K})$. The proof that $\left(\Delta^{\phi, T} (K\times \widetilde{K})\right)^\circ \neq \varnothing$ follows by using induction on the depth of trees in the same proof in Theorem \ref{thm:tree} with $\Delta^{\phi, T}$ in place of $\Delta^T$. This completes the proof.
\end{proof}

\section{Open questions}\label{sec:open-questions}
\begin{question}\label{q:1}
If $K\subset \R$ is a Cantor set, is there exists a Cantor set $\widetilde{K}$ such that $\Delta^T(K\times \widetilde{K})$ has nonempty interior for all finite trees $T$?
\end{question}
If the answer to Question \ref{q:1} is yes, then this provides a way to construct a Cantor set $C\subset \R^2$ with $\dim_{\rm H}=1$ such that $\Delta^T (C)$ has nonempty interior for all finite trees $T$. Note that by taking finite unions in Theorem \ref{thm:tree}, the answer to Question \ref{q:1} is yes for any finite collection of trees.

\begin{question}\label{q:2}
    For an odd integer $d\geq 3$ and for a finite tree $T$ with more than two vertices, can we construct $C\subset \R^d$ such that $\dim_{\rm H} (C)=\frac{d}{2}$ and $\Delta^T (C)$ has nonempty interior?
\end{question}
To our knowledge, it is unknown how to construct a Cantor set $C\subset \R^d$ for odd $d\geq 3$ such that $\dim_{\rm H} (C)=\frac{d}{2}$ and the pinned distance set $\Delta_x(C)$ has nonempty interior for some $x\in C$. We finish this section by asking several possible generalizations of tree distace sets.
\begin{question}
    Can we generalize our results (specifically, Theorem \ref{thm:tree-Rd} and Corollary \ref{cor:tree-hausdorff-dimension-zero-Rd}) to more general class of graphs?
\end{question}
These may include infinite chains and trees or graphs with loops, such as polygons. We may also impose requirements on the lengths of the distances like constant-gaps (see \cite{McDonald_Taylor_2023_infinite_chains}), such as equilateral triangles.

\smallskip

\noindent{\bf Acknowledgments.} K.T. is supported in part by the Simons Foundation Grant GR137264. The authors thank Eyvindur Palsson for helpful discussions related to this article.

\bibliographystyle{plain} 
\bibliography{refs}

\begin{thebibliography}{10}

\bibitem{BIT}
Michael Bennett, Alexander Iosevich, and Krystal Taylor.
\newblock Finite chains inside thin subsets of {$\Bbb{R}^d$}.
\newblock {\em Anal. PDE}, 9(3):597--614, 2016.

\bibitem{BoonePalsson}
Zack Boone and Eyvindur~Ari Palsson.
\newblock A pinned {M}attila-{S}j\"{o}lin type theorem for product sets.
\newblock arXiv:2210.00675, 2022.

\bibitem{BFOP}
Tainara Borges, Benjamin Foster, Yumeng Ou, and Eyvindur Palsson.
\newblock Nonempty interior of pinned distance and tree sets.
\newblock arXiv:2503.15709, 2025.

\bibitem{Falc85paper}
Kenneth Falconer.
\newblock On the {H}ausdorff dimensions of distance sets.
\newblock {\em Mathematika}, 32(2):206--212, 1985.

\bibitem{falconer-fractal-geometry}
Kenneth Falconer.
\newblock {\em Fractal geometry}.
\newblock John Wiley \& Sons, Inc., Hoboken, NJ, second edition, 2003.
\newblock Mathematical foundations and applications.

\bibitem{GIOW}
Larry Guth, Alex Iosevich, Yumeng Ou, and Hong Wang.
\newblock On {F}alconer's distance set problem in the plane.
\newblock {\em Invent. Math.}, 219(3):779--830, 2020.

\bibitem{IMT12}
Alex Iosevich, Mihalis Mourgoglou, and Krystal Taylor.
\newblock On the {M}attila-{S}j\"olin theorem for distance sets.
\newblock {\em Ann. Acad. Sci. Fenn. Math.}, 37(2):557--562, 2012.

\bibitem{IT}
Alex Iosevich and Krystal Taylor.
\newblock Finite trees inside thin subsets of {$\Bbb R^d$}.
\newblock In {\em Modern methods in operator theory and harmonic analysis}, volume 291 of {\em Springer Proc. Math. Stat.}, pages 51--56. Springer, Cham, 2019.

\bibitem{Jiang}
Kan Jiang.
\newblock Obtaining an explicit interval for a nonlinear {N}ewhouse thickness theorem.
\newblock {\em Math. Z.}, 301(1):1011--1037, 2022.

\bibitem{jung-lai-interior-of-certain}
Yeonwook Jung and Chun-Kit Lai.
\newblock Interior of certain sums and continuous images of very thin {C}antor sets.
\newblock arXiv:2410.01267, 2024.

\bibitem{MS}
Pertti Mattila and Per Sj\"olin.
\newblock Regularity of distance measures and sets.
\newblock {\em Math. Nachr.}, 204:157--162, 1999.

\bibitem{McDonald-Taylor-finite-point-configurations}
Alex McDonald and Krystal Taylor.
\newblock Finite point configurations in products of thick {C}antor sets and a robust nonlinear {N}ewhouse gap lemma.
\newblock {\em Math. Proc. Cambridge Philos. Soc.}, 175(2):285--301, 2023.

\bibitem{McDonald_Taylor_2023_infinite_chains}
Alex McDonald and Krystal Taylor.
\newblock Infinite constant gap length trees in products of thick {C}antor sets.
\newblock {\em Proceedings of the Royal Society of Edinburgh: Section A Mathematics}, page 1–12, 2023.

\bibitem{Moschovakis}
Yiannis~N. Moschovakis.
\newblock {\em Descriptive set theory}, volume 155 of {\em Mathematical Surveys and Monographs}.
\newblock American Mathematical Society, Providence, RI, second edition, 2009.

\bibitem{OT22}
Yumeng Ou and Krystal Taylor.
\newblock Finite point configurations and the regular value theorem in a fractal setting.
\newblock {\em Indiana Univ. Math. J.}, 71(4):1707--1761, 2022.

\bibitem{Palis1993HyperbolicityAS}
Jacob Palis and Floris Takens.
\newblock Fractal dimensions and infinitely many attractors.
\newblock In {\em Hyperbolicity and sensitive chaotic dynamics at homoclinic bifurcations : fractal dimensions and infinitely many attractors}, volume~35. Cambridge University Press, 1993.

\bibitem{PS}
Yuval Peres and Wilhelm Schlag.
\newblock Smoothness of projections, {B}ernoulli convolutions, and the dimension of exceptions.
\newblock {\em Duke Math. J.}, 102(2):193--251, 2000.

\bibitem{ST2025}
Samantha Sandberg and Krystal Taylor.
\newblock Triangles in the plane and arithmetic progressions in thick compact subsets of $\mathbb{R}^d$.
\newblock arXiv:2210.00675, 2025.

\bibitem{simon_taylor}
Károly Simon and Krystal Taylor.
\newblock Interior of sums of planar sets and curves.
\newblock {\em Mathematical Proceedings of the Cambridge Philosophical Society}, 168, 2017.

\bibitem{yavicoli_patterns}
Alexia Yavicoli.
\newblock Patterns in thick compact sets.
\newblock {\em Israel Journal of Mathematics}, 244, 2021.

\bibitem{Yavicoli_Survey}
Alexia Yavicoli.
\newblock A survey on {N}ewhouse thickness, fractal intersections and patterns.
\newblock {\em Mathematical and Computational Applications}, 2022.

\bibitem{Yavicoli_Gap_Lemma_Rd}
Alexia Yavicoli.
\newblock {Thickness and a Gap Lemma in $\mathbb{R}^d$}.
\newblock {\em International Mathematics Research Notices}, 2023(19):16453--16477, 2022.

\end{thebibliography}

\end{document}